\documentclass[11pt]{article}

\usepackage[margin=1.15in]{geometry}
\usepackage[T1]{fontenc}
\usepackage{amssymb,amsmath,amsthm}
\usepackage{booktabs}
\usepackage[hidelinks]{hyperref}

\newtheorem{theorem}{Theorem}
\newtheorem{lemma}{Lemma}
\newtheorem{proposition}{Proposition}
\newtheorem{corollary}{Corollary}

\theoremstyle{definition}
\newtheorem{definition}{Definition}
\newtheorem{conjecture}{Conjecture}
\newtheorem{remark}{Remark}

\newcommand{\rad}{\operatorname{rad}}

\title{The One-Power Erd\H{o}s Conjecture for Primes $p \equiv 7 \pmod{36}$}

\author{Shen Harman\\
\small Paquin Institute of Mathematics\\
\small \texttt{shen@paquinmath.org}
\and
Dana Paquin\\
\small Department of Mathematics, Stanford University, Stanford, California\\
\small \texttt{dpaquin@stanford.edu}}

\date{\today}

\begin{document}

\maketitle

\begin{abstract}
A conjecture of Erd\H{o}s asserts that every odd integer $n > 1$ can be written as $n = 2^a + m$, where $a \ge 0$ is an integer and $m$ is a squarefree positive integer. We study this conjecture for primes $p \equiv 7 \pmod{36}$. Writing $x = (p-1)/6$, the candidate remainders take the shapes $p - 1 = 6x$, $p-2 = 6x-1$, $p-4 = 3(2x-1)$, and $p-8 = 6x-7$, in which the primes $2$ and $3$ are completely controlled: every obstruction to squarefreeness comes from a prime $q \ge 5$. For such $q$ a separation principle applies --- $q^2$ divides at most one of the remainders, since two of them differ by at most $7 < 25$ --- and it persists across any window of exponents of width less than $20$. Separation collapses the square-divisor problem for the quartic $x(6x-1)(2x-1)(6x-7)$ into independent linear problems, and we deduce unconditionally that the primes of the class representable with $a \le 3$ have relative density $1 - \rho_4$, where $\rho_4 = 1.8637 \times 10^{-5}$ is an explicit Euler-type constant; in particular, infinitely many primes of the class are not so representable. Computationally, we verify the conjecture for all primes of the class below $10^8$, always with $a \le 5$, and we exhibit the eleven primes below $10^8$ requiring $a = 5$. An appendix corrects the error bound in a lemma of Hercher, whose conclusions are unaffected.
\end{abstract}


\section{Introduction}\label{sec:intro}

In 1950, Erd\H{o}s \cite{erdos1950} proved that there are infinitely many positive integers not expressible as the sum of a prime and a power of two, and in later problem collections \cite{erdosgraham, erdos1997} he proposed a relaxation that remains open: every odd integer $n > 1$ should be the sum of a squarefree number and a power of two. The problem appears as Problem \#11 on Bloom's list of Erd\H{o}s problems \cite{bloom}. We study the following one-power form of the conjecture, restricted to primes.

\begin{conjecture}[Erd\H{o}s, one-power version]\label{conj:onepower}
Every prime $p$ can be written as $p = 2^a + m$, where $a \ge 0$ is an integer and $m$ is a squarefree positive integer.
\end{conjecture}

Computational support for the family of Erd\H{o}s squarefree-plus-power-of-two problems is extensive: Odlyzko verified the odd-integer conjecture up to $10^7$ (as reported by Granville and Soundararajan \cite{gransound}), McCranie up to $1.4 \times 10^9$ (see Guy \cite[\S A19]{guy}), and recently Hercher \cite{hercher} up to $2^{50}$ using a GPU-parallelized sieve. On the theoretical side, Granville and Soundararajan \cite{gransound} connected the conjecture to Wieferich primes and to the order of $2$ modulo $p^2$, and Granville \cite{granville98} showed that the $abc$ conjecture implies that squarefree polynomial values appear with their expected density. That conditional input is not required here, since the polynomials we sieve are products of distinct linear forms.

We treat the residue class $p \equiv 7 \pmod{36}$. Writing $p = 36k+7$ and $x = 6k+1$, the first four candidate remainders take the shapes
\[
p - 1 = 6x, \qquad p - 2 = 6x - 1, \qquad p - 4 = 3(2x-1), \qquad p - 8 = 6x - 7.
\]
Our contributions are as follows.

\begin{enumerate}
\item \textit{Structure.} We prove that $x$, $6x-1$, $2x-1$, $6x-7$ are pairwise coprime apart from a single controlled factor of $7$, and deduce a \emph{strong separation} theorem: for every prime $q \ge 5$, the square $q^2$ divides at most one of the four quantities (Theorem~\ref{thm:coprime}, Corollary~\ref{cor:separation}). Consequently a \emph{complete failure} --- all four candidate remainders non-squarefree --- forces $(q_1q_2q_3q_4)^2 \mid P(x)$ for four distinct primes $q_i \ge 5$, where $P(x) = x(6x-1)(2x-1)(6x-7)$ is a squarefree quartic (Proposition~\ref{prop:poly}). Separation persists across any window of admissible exponents of width less than $20$ (Lemma~\ref{lem:ord}, Proposition~\ref{prop:extended}), while $p - 16$ is \emph{never} squarefree, being always divisible by $9$ (Lemma~\ref{lem:excluded}).
\item \textit{Density theorems, unconditional.} Strong separation collapses the square-divisor problem for the quartic $P$ into four independent linear problems, with no genuinely quartic behaviour. Over integer arguments this yields the exact density $\rho_4^{\mathbb{Z}}$ of complete failures (Theorem~\ref{thm:integer}). At prime arguments we prove squarefree equidistribution for \emph{every} subset of the four forms: the single-form cases by the method of Mirsky \cite{mirsky} (Theorem~\ref{thm:singleform}), and the multi-form cases by a truncated M\"obius expansion whose cross terms are controlled in mean by fourth-moment estimates (Theorem~\ref{thm:sf}). Consequently the primes admitting a representation with $a \le 3$ have relative density exactly $1 - \rho_4$, where $\rho_4$ is an explicit Euler-type constant, $\rho_4 \approx 1.8637 \times 10^{-5}$ (Theorem~\ref{thm:conditional}), and complete failures occur infinitely often, with positive relative density (Corollary~\ref{cor:infinitude}). These are density statements, not universal ones.
\item \textit{Computation.} We verify Conjecture~\ref{conj:onepower} for all $p \equiv 7 \pmod{36}$ with $p < 10^8$, and determine the complete failures in this range: there are exactly eleven, listed in Table~\ref{tab:failures}. Every prime in the class below $10^8$ is resolved with exponent $a \in \{0,1,2,3,5\}$ (Theorem~\ref{thm:computational}).
\item \textit{The prime-argument correction, and what remains conjectural.} Because our arguments run over \emph{primes}, the correct local density for the event $q^2 \mid p - 2^a$ is $1/\varphi(q^2) = 1/(q(q-1))$, not $1/q^2$. The resulting predictions --- backed for every fixed exponent window by Theorem~\ref{thm:sf} --- match the observed counts at every depth (e.g., $8.95$ expected complete failures below $10^8$ against $11$ observed), while the naive $1/q^2$ model is excluded at $33$ standard deviations. The statements that remain conjectural are the \emph{universal} ones: that every prime of the class has some usable exponent (Conjecture~\ref{conj:class}). In that direction we show unconditionally that the exceptional set of the full conjecture has relative density below $10^{-26}$ (Remark~\ref{rem:cofinite}), and that the smallest usable exponent is unbounded over the class (Proposition~\ref{prop:unbounded}).
\item \textit{A correction to the literature.} While reconstructing the proofs of Hercher's paper \cite{hercher}, we found that the error bound $|e(x)| < 18x^3$ in \cite[Lemma 2]{hercher} cannot hold as stated: the error term has true cubic coefficient $(k^3-k)/3$, which equals $2660$ at $k = 20$. We prove a corrected bound and verify that all of the conclusions of \cite{hercher} are unaffected (Appendix~\ref{sec:hercher}).
\end{enumerate}

The remainder of this paper is organized as follows. Section~\ref{sec:background} fixes notation, records the reduction to linear forms, and collects the known results we rely on. Section~\ref{sec:structure} proves the structural results: pairwise coprimality, strong separation, the primary failure lemma, and the polynomial reformulation. Section~\ref{sec:conditional} proves the density theorems --- single and multiple forms at prime arguments, and the integer-argument analogue --- all unconditionally. Section~\ref{sec:computation} describes the computations to $10^8$ and presents the evidence in tables. Section~\ref{sec:heuristic} develops the prime-corrected heuristic model and states our conjectures. Section~\ref{sec:conclusion} summarizes and lists directions for future work, and Appendix~\ref{sec:hercher} records the correction to \cite[Lemma 2]{hercher}.

\section{Background and Notation}\label{sec:background}

\begin{definition}\label{def:squarefree}
A positive integer $m$ is \textit{squarefree} if no prime square divides $m$; equivalently, $m = \rad(m)$, where $\rad(m)$ denotes the product of the distinct primes dividing $m$. For a polynomial $F$, $\rad(F)$ denotes the product of the distinct irreducible factors of $F$.
\end{definition}

Throughout, $p$ denotes a prime with $p \equiv 7 \pmod{36}$, and $q$ denotes a prime obstructing squarefreeness. Since $36 = 4 \cdot 9$, the Chinese Remainder Theorem gives the equivalent description
\[
p \equiv 3 \pmod 4 \qquad \text{and} \qquad p \equiv 7 \pmod 9.
\]
We write $\pi(N; 36, 7)$ for the number of primes $p \le N$ with $p \equiv 7 \pmod{36}$, and we measure densities of subsets of this class relative to $\pi(N; 36, 7)$.

\subsection{The Key Substitution}

\begin{definition}\label{def:x}
For $p = 36k + 7$ with $k \ge 0$, define $x = 6k + 1$, so that $p = 6x + 1$. Note that $x$ is odd and $x \equiv 1 \pmod 3$; in particular $\gcd(x, 6) = 1$.
\end{definition}

\begin{lemma}[Reduction to linear forms]\label{lem:forms}
Let $p = 36k+7$ and $x = 6k+1$. Then
\[
p - 1 = 6x, \qquad
p - 2 = 6x - 1, \qquad
p - 4 = 3(2x - 1), \qquad
p - 8 = 6x - 7,
\]
and more generally $p - 2^a = 6x - (2^a - 1)$ for every $a \ge 0$ with $2^a < p$. Moreover, for every prime $q \ge 5$,
\[
q^2 \mid p - 1 \text{ if and only if } q^2 \mid x,
\qquad
q^2 \mid p - 4 \text{ if and only if } q^2 \mid 2x - 1,
\]
and $p - 1$ (resp.\ $p-2$, $p-4$, $p-8$) is squarefree if and only if $x$ (resp.\ $6x-1$, $2x-1$, $6x-7$) is squarefree.
\footnote{Lean 4 formalizations of the four remainder identities, squarefree equivalences, and resulting representation criterion in Section 2.1 are available at https://github.com/shenharman/erdos-7-mod-36.}
\end{lemma}

\begin{proof}
Since $p = 6x + 1$, the identity $p - 2^a = 6x - (2^a - 1)$ is immediate; the displayed cases are $a = 0, 1, 2, 3$ (for $a = 2$ note $6x - 3 = 3(2x-1)$). For the squarefree equivalences: $p - 1 = 6x = 2 \cdot 3 \cdot x$ with $\gcd(x, 6) = 1$, and $6$ is squarefree, so $p-1$ is squarefree if and only if $x$ is; the prime squares $4$ and $9$ never divide $6x$. Similarly $p - 4 = 3(2x-1)$ with $2x - 1$ odd and $2x - 1 \equiv 1 \pmod 3$, so $\gcd(3, 2x-1) = 1$ and $p - 4$ is squarefree if and only if $2x-1$ is. Finally $6x - 1$ and $6x - 7$ are odd and coprime to $3$, so their squarefreeness is unaffected by the primes $2, 3$; obstructions to any of the four forms must come from primes $q \ge 5$.
\end{proof}

Thus the one-power conjecture for $p$ with exponent $a \le 3$ holds if and only if at least one of
\[
x, \qquad 6x - 1, \qquad 2x - 1, \qquad 6x - 7
\]
is squarefree. We call $p - 1$ the \emph{primary} remainder and $p-2$, $p-4$, $p-8$ the \emph{fallback} remainders, and we speak of a primary or fallback \emph{failure} when the remainder in question is not squarefree. This rigidity --- all four candidate remainders determined by the single quantity $x$ --- is the central structural feature of the class $p \equiv 7 \pmod{36}$, and we exploit it throughout.

\begin{definition}[Complete failure]\label{def:completefailure}
A prime $p \equiv 7 \pmod{36}$ is a \emph{complete failure} if none of $p-1$, $p-2$, $p-4$, $p-8$ is squarefree; equivalently, by Lemma~\ref{lem:forms}, if all four of $x$, $6x-1$, $2x-1$, $6x-7$ are non-squarefree at $x = (p-1)/6$.
\end{definition}

Not every exponent is available, however:

\begin{lemma}[Excluded exponents]\label{lem:excluded}
Let $p \equiv 7 \pmod{36}$ and $a \ge 0$. Then $9 \mid p - 2^a$ if and only if $a \equiv 4 \pmod 6$. In particular $p - 16 = 9(4k - 1)$ is never squarefree, and the admissible exponents are exactly those with $a \not\equiv 4 \pmod 6$. Moreover $4 \nmid p - 2^a$ for every $a$, so for admissible $a$ every prime square obstructing $p - 2^a$ satisfies $q \ge 5$.
\end{lemma}

\begin{proof}
The order of $2$ modulo $9$ is $6$, with $2^0, \dots, 2^5 \equiv 1, 2, 4, 8, 7, 5 \pmod 9$. Since $p \equiv 7 \pmod 9$, we have $9 \mid p - 2^a$ if and only if $2^a \equiv 7 \pmod 9$, that is, if and only if $a \equiv 4 \pmod 6$. Reducing $p - 2^a$ modulo $9$ for $a \equiv 0, 1, \dots, 5 \pmod 6$ gives the residues $6, 5, 3, 8, 0, 2$ respectively. Hence $3$ exactly divides $p - 2^a$ when $a \equiv 0, 2 \pmod 6$ --- in particular $3 \parallel p - 1 = 6x$, since $x \equiv 1 \pmod 3$, and $3 \parallel p - 4 = 3(2x-1)$ --- while $3 \nmid p - 2^a$ when $a \equiv 1, 3, 5 \pmod 6$. In no admissible case does $9$ divide $p - 2^a$. For the prime $2$, if $a = 0$ then $p - 1 \equiv 2 \pmod 4$, because $p - 1 = 6x$ with $x$ odd, and if $a \ge 1$ then $p - 2^a$ is odd.
\end{proof}

\subsection{Known Results}

We will use the following results from the literature.

\begin{theorem}[Mirsky \cite{mirsky}]\label{thm:mirsky}
The set of primes $p$ for which $p - 1$ is squarefree has relative density
$\prod_{q} \left(1 - \frac{1}{q(q-1)}\right) = 0.3739558\dots$ (Artin's constant) within the primes.
\end{theorem}

The appearance of $\varphi(q^2) = q(q-1)$ in place of $q^2$ in Theorem~\ref{thm:mirsky} is the fundamental prime-argument correction: among primes, the event $q^2 \mid p - c$ (for $q \nmid c$) is the Dirichlet condition $p \equiv c \pmod{q^2}$, which holds for a proportion $1/\varphi(q^2)$ of primes, not $1/q^2$. This correction drives the heuristic model of Section~\ref{sec:heuristic}.

For squarefree values of polynomials of higher degree the situation is different, and conditional results are the state of the art: Estermann \cite{estermann} settled $n^2+1$, Granville \cite{granville98} showed that the $abc$ conjecture yields the expected density for any squarefree polynomial, and Mason \cite{mason} and Stothers \cite{stothers} give the function-field analogue of $abc$; see Granville and Tucker \cite{grantucker} for a survey. None of this is needed here. Every polynomial we sieve is a product of distinct linear forms, and for such a product the squarefree density is elementary: beyond the finitely many $q$ dividing a pairwise resultant, $q^2$ can divide at most one factor, and each linear factor is handled by the trivial tail bound $\sum_{q > z}(X/q^2 + 1)$.

\section{The Algebraic Structure}\label{sec:structure}

The results of this section are unconditional and elementary. Two of them deserve a word on scope before we begin. The separation principle of Corollary~\ref{cor:separation} is not special to this residue class: it follows from the size of the differences $2^b - 2^a$ and holds for every prime, as the proof shows. What the class contributes is the control of the primes $2$ and $3$ in Lemmas~\ref{lem:forms} and~\ref{lem:excluded}, which is what confines obstructions to $q \ge 5$ and so makes the size argument available in the first place. Theorem~\ref{thm:coprime} records the pairwise gcds explicitly, since they are needed for the polynomial reformulation, but the separation itself does not depend on them.

\begin{theorem}[Pairwise coprimality]\label{thm:coprime}
Let $x = 6k+1$ for an integer $k \ge 0$. The six pairwise gcds of the quantities $x$, $6x-1$, $2x-1$, $6x-7$ satisfy
\begin{align*}
\gcd(x, 6x{-}1) = \gcd(x, 2x{-}1) &= \gcd(6x{-}1, 2x{-}1) \\
= \gcd(6x{-}1, 6x{-}7) &= \gcd(2x{-}1, 6x{-}7) = 1,
\end{align*}
while $\gcd(x, 6x-7) = \gcd(x, 7) \in \{1, 7\}$, with value $7$ exactly when $7 \mid x$. Moreover $7 \mid x$ if and only if $7 \mid 6x - 7$, but $49$ divides at most one of $x$ and $6x-7$.
\end{theorem}

\begin{proof}
We compute each gcd by eliminating $x$.
(i) Any common divisor of $x$ and $6x - 1$ divides $6x - (6x-1) = 1$.
(ii) Any common divisor of $x$ and $2x - 1$ divides $2x - (2x-1) = 1$.
(iii) Any common divisor $d$ of $x$ and $6x - 7$ divides $6x - (6x - 7) = 7$, so $d \in \{1, 7\}$; and since $6x - 7 \equiv -x \pmod 7$, we have $7 \mid x$ if and only if $7 \mid 6x-7$.
(iv) Any common divisor of $6x-1$ and $2x-1$ divides $(6x-1) - 3(2x-1) = 2$; both quantities are odd, so the gcd is $1$.
(v) Any common divisor of $6x - 1$ and $6x - 7$ divides their difference $6$; both are odd and coprime to $3$, so the gcd is $1$.
(vi) Any common divisor of $2x-1$ and $6x-7$ divides $3(2x-1) - (6x - 7) = 4$; both are odd, so the gcd is $1$.
Finally, if $49 \mid x$ and $49 \mid 6x - 7$, then $49 \mid 6x - (6x-7) = 7$, a contradiction; so $49$ divides at most one of the pair.
\end{proof}

\begin{corollary}[Strong separation]\label{cor:separation}
For every prime $q \ge 5$, the square $q^2$ divides at most one of the four quantities $x$, $6x-1$, $2x-1$, $6x-7$.
\end{corollary}

\begin{proof}
If $q \ne 7$, then by Theorem~\ref{thm:coprime} the prime $q$ itself divides at most one of the four quantities, so a fortiori $q^2$ does. If $q = 7$, the only pair with a possible common factor is $(x, 6x-7)$, and the final assertion of Theorem~\ref{thm:coprime} shows $49$ divides at most one of them.
\end{proof}

\begin{remark}
The size estimate used above --- that two of the four remainders differ by at most $7$, while $q^2 \ge 25$ --- is the whole of the argument, and it requires nothing of $p$ beyond $q \ge 5$. The same estimate reappears in the proof of Theorem~\ref{thm:sf}, where it is what makes non-coprime tuples of moduli vanish. Lemma~\ref{lem:primaryfailure} is then the special case in which the obstructed form is $x$ itself.
\end{remark}

\begin{lemma}[Primary Failure Lemma]\label{lem:primaryfailure}
Let $x = 6k+1$ and let $q \ge 5$ be prime with $q^2 \mid x$. Then $q^2 \nmid 6x - 1$, $q^2 \nmid 2x - 1$, and $q^2 \nmid 6x - 7$.
\end{lemma}

\begin{proof}[First proof (direct computation)]
If $q^2 \mid x$ then modulo $q^2$ we have $6x - 1 \equiv -1$, $2x - 1 \equiv -1$, and $6x - 7 \equiv -7$. Since $q \ge 5$, the modulus $q^2 \ge 25$ divides none of $1, 1, 7$.
\end{proof}

\begin{proof}[Second proof (via strong separation)]
Since $q^2 \mid x$, Corollary~\ref{cor:separation} immediately gives that $q^2$ divides none of the other three quantities.
\end{proof}

\begin{remark}
The first proof is self-contained and makes visible \emph{why} the residues $-1, -1, -7$ are harmless (they are too small in absolute value to be multiples of $q^2 \ge 25$, and nonzero). The second proof is shorter and generalizes better: it identifies the lemma as the special case ``$q^2$ hits the form $x$'' of the uniform statement that $q^2$ can hit at most one form, which applies symmetrically to fallback failures as well.
\end{remark}

\begin{proposition}[Polynomial reformulation]\label{prop:poly}
If $p$ is a complete failure (Definition~\ref{def:completefailure}), then there exist four \emph{distinct} primes $q_1, q_2, q_3, q_4 \ge 5$ with
\[
q_1^2 \mid x, \qquad q_2^2 \mid 6x-1, \qquad q_3^2 \mid 2x-1, \qquad q_4^2 \mid 6x-7,
\]
and consequently
\[
(q_1 q_2 q_3 q_4)^2 \,\big|\, P(x), \qquad \text{where } P(x) = x(6x-1)(2x-1)(6x-7).
\]
Explicitly, $P(x) = 72x^4 - 132x^3 + 62x^2 - 7x$. The quartic $P$ has the four distinct rational roots $0, \tfrac16, \tfrac12, \tfrac76$; it is squarefree in $\mathbb{Q}[x]$, with discriminant $2^{14} \cdot 3^4 \cdot 7^2 = 65{,}028{,}096$. Moreover $(p-1)(p-2)(p-4)(p-8) = 18\,P(x)$.
\end{proposition}

\begin{proof}
By Lemma~\ref{lem:forms}, obstructions to each form come only from primes $q \ge 5$, so each non-squarefree form is divisible by $q_i^2$ for some prime $q_i \ge 5$. The $q_i$ are pairwise distinct by Corollary~\ref{cor:separation} (a single $q$ cannot square-divide two forms). Since the $q_i^2$ are powers of distinct primes each dividing a factor of the product $P(x)$, their product $(q_1q_2q_3q_4)^2$ divides $P(x)$. The expansion is a direct computation:
we have $(6x-1)(2x-1) = 12x^2 - 8x + 1$ and $x(6x-7) = 6x^2 - 7x$, whence $P(x) = (12x^2-8x+1)(6x^2-7x) = 72x^4 - 132x^3 + 62x^2 - 7x$; the roots of the factors are $0, \tfrac16, \tfrac12, \tfrac76$, which are distinct, so $P$ is squarefree over $\mathbb{Q}$. The discriminant evaluates to $72^6 \prod_{i<j}(r_i - r_j)^2 = 2^{14} \cdot 3^4 \cdot 7^2$. Finally $(p-1)(p-2)(p-4)(p-8) = 6x \cdot (6x-1) \cdot 3(2x-1) \cdot (6x-7) = 18 P(x)$.
\end{proof}

The separation phenomenon is not confined to the first four remainders. Its true extent is governed by the multiplicative order of $2$ modulo prime squares:

\begin{lemma}[Order of $2$ modulo prime squares]\label{lem:ord}
For every prime $q \ge 5$, the multiplicative order of $2$ modulo $q^2$ is at least $20$, with equality if and only if $q = 5$.
\end{lemma}

\begin{proof}
If the order is $m$, then $q^2 \mid 2^m - 1$; so it suffices to check that no $2^m - 1$ with $m \le 19$ is divisible by the square of a prime $\ge 5$. The factorizations are
\begin{align*}
&1,\ 3,\ 7,\ 3\cdot5,\ 31,\ 3^2\cdot7,\ 127,\ 3\cdot5\cdot17,\ 7\cdot73,\ 3\cdot11\cdot31,\ 23\cdot89,\ 3^2\cdot5\cdot7\cdot13,\\
&8191,\ 3\cdot43\cdot127,\ 7\cdot31\cdot151,\ 3\cdot5\cdot17\cdot257,\ 131071,\ 3^3\cdot7\cdot19\cdot73,\ 524287
\end{align*}
for $m = 1, \dots, 19$; the only prime squares occurring are powers of $3$ (at $m = 6, 12, 18$), which are excluded since $q \ge 5$. Hence the order is at least $20$ for every $q \ge 5$, and the factorization $2^{20} - 1 = 3 \cdot 5^2 \cdot 11 \cdot 31 \cdot 41$ shows that the order of $2$ modulo $25$ is exactly $20$, while for $q > 5$ the order exceeds $20$ (as $5^2$ is the only square $\ge 25$ dividing $2^{20}-1$).
\end{proof}

\begin{proposition}[Windowed separation]\label{prop:extended}
Let $A$ be any set of admissible exponents contained in an interval of length at most $19$. Then for every prime $q \ge 5$, the square $q^2$ divides at most one of the remainders $p - 2^a$, $a \in A$. In particular, $q^2$ divides at most one of the six quantities
\[
x, \qquad 6x - 1, \qquad 2x - 1, \qquad 6x - 7, \qquad 6x - 31, \qquad 2x - 21,
\]
corresponding via Lemma~\ref{lem:forms} to the admissible remainders $p - 2^a$, $a \in \{0, 1, 2, 3, 5, 6\}$ (note $p - 64 = 3(2x - 21)$).
\end{proposition}

\begin{proof}
If $q^2 \mid p - 2^a$ and $q^2 \mid p - 2^b$ with $a < b$ in $A$, then $q^2$ divides the difference $2^b - 2^a = 2^a(2^{b-a} - 1)$, and since $q$ is odd, $q^2 \mid 2^{b-a} - 1$. Thus the order of $2$ modulo $q^2$ divides $b - a \le 19$, contradicting Lemma~\ref{lem:ord}. The six-form case follows since $\{0,1,2,3,5,6\}$ has width $6$, together with the normalizations $q^2 \mid p - 1$ if and only if $q^2 \mid x$, $q^2 \mid p-4$ if and only if $q^2 \mid 2x-1$, and $q^2 \mid p - 64$ if and only if $q^2 \mid 2x - 21$ for $q \ge 5$.
\end{proof}

\begin{remark}[Wieferich-type boundary]\label{rem:wieferich}
By Lemma~\ref{lem:ord}, separation fails first when exponents exactly $20$ apart are compared, and then only at $q = 5$: both $p - 2^a$ and $p - 2^{a+20}$ can be divisible by $25$ (for $p$ in an appropriate class modulo $25$). This is the same threshold at which Hercher's heuristic computation \cite{hercher} changes character (his boundary case $2^{20}-1$), and the condition $q^2 \mid 2^d - 1$ is a Wieferich-type condition in the sense of \cite{gransound}.
\end{remark}

\section{Density Theorems}\label{sec:conditional}

We now quantify how often each representation succeeds. The correct local densities are dictated by Dirichlet's theorem, as follows.

\begin{lemma}[Local densities at prime arguments]\label{lem:local}
Fix a prime $q \ge 5$ and one of the four forms $f \in \{x,\, 6x-1,\, 2x-1,\, 6x-7\}$. The set of primes $p \equiv 7 \pmod{36}$ with $q^2 \mid f(x_p)$ (where $x_p = (p-1)/6$) is exactly the set of primes in a single residue class modulo $36q^2$ that is admissible (coprime to the modulus); by Dirichlet's theorem it has relative density
\[
\frac{1}{\varphi(q^2)} = \frac{1}{q(q-1)}
\]
within the class $p \equiv 7 \pmod{36}$. Moreover, for distinct forms these residue classes are distinct, so the corresponding events are mutually exclusive at the level of a single prime $q$.
\end{lemma}

\begin{proof}
The condition $q^2 \mid f(x)$ determines $x$ modulo $q^2$ (the leading coefficient of each form is invertible modulo $q$), hence determines $p = 6x + 1$ in a single class modulo $6q^2$, and together with $p \equiv 7 \pmod{36}$, in a single class modulo $36q^2$ by CRT ($q \ge 5$ is coprime to $6$). This class is coprime to the modulus: $p \equiv 7 \pmod{36}$ is coprime to $36$, and $p \equiv 2^a \pmod{q^2}$ for the relevant $a \in \{0,1,2,3\}$ is coprime to $q$ since $q \ge 5 > 2$. Dirichlet's theorem gives density $1/\varphi(36q^2)$ among all primes, i.e., relative density $\varphi(36)/\varphi(36q^2) = 1/\varphi(q^2)$ within the class. Distinctness of the classes for distinct forms is a restatement of Corollary~\ref{cor:separation}: if two forms produced the same class modulo $36q^2$, any prime in it would have $q^2$ dividing two forms.
\end{proof}

For a single form, the expected local--global behavior is not merely a hypothesis; it is a theorem, proved by adapting Mirsky's argument \cite{mirsky} to our progression.

\begin{theorem}[Single forms, unconditional]\label{thm:singleform}
Let $c \in \{1, 2, 4, 8\}$. The set of primes $p \equiv 7 \pmod{36}$ for which $p - c$ is squarefree has relative density
\[
e_1 = \prod_{q \ge 5}\left(1 - \frac{1}{q(q-1)}\right) = 0.8974939526861\ldots
\]
within the class. The same holds with $c = 2^a$ for any fixed admissible exponent $a$.
\end{theorem}

\begin{proof}
Write $\pi_c(N; m)$ for the number of primes $p \le N$ with $p \equiv 7 \pmod{36}$ and $p \equiv c \pmod{m}$. By Lemma~\ref{lem:excluded}, no square of $2$ or $3$ divides $p - c$, so by inclusion--exclusion over square divisors,
\[
\#\{p \le N \text{ in the class} : p - c \text{ squarefree}\} \;=\; \sum_{d \le \sqrt{N}} \mu(d)\, \pi_c(N; d^2),
\]
and only $d$ coprime to $6$ contribute: for $2 \mid d$ or $3 \mid d$ the congruences $p \equiv 7 \pmod{36}$, $p \equiv c \pmod{d^2}$ are incompatible, so $\pi_c(N; d^2) = 0$. Set $y = (\log N)^{2}$ and split the sum. For $d \le y$ with $(d, 6) = 1$, the two congruences define a single admissible class modulo $36d^2$, and the Siegel--Walfisz theorem gives
\[
\pi_c(N; d^2) = \frac{\pi(N)}{\varphi(36)\,\varphi(d^2)} + O\!\left(N \exp(-c_0\sqrt{\log N})\right)
\]
uniformly. For $d > y$ we use the trivial bound $\pi_c(N; d^2) \le N/(36d^2) + 1$, so that range contributes $O(N/y + \sqrt{N})$. Since $\sum_{(d,6)=1} \mu(d)/\varphi(d^2) = \prod_{q \ge 5}(1 - 1/\varphi(q^2))$ converges absolutely (as $\varphi(d^2) = d\varphi(d) \gg d^2/\log\log d$), with tail $O(1/y)$, we obtain
\[
\#\{p \le N \text{ in the class} : p - c \text{ squarefree}\} = \frac{\pi(N)}{12} \prod_{q \ge 5}\left(1 - \frac{1}{q(q-1)}\right) + o(\pi(N)),
\]
and dividing by $\pi(N; 36, 7) \sim \pi(N)/12$ gives the theorem. The argument uses nothing about $c$ beyond admissibility, so it applies to any fixed admissible $2^a$.
\end{proof}

The same conclusion holds for several forms at once. By Lemma~\ref{lem:local}, the result below asserts precisely the expected local--global behavior, the local factor at each $q$ being $1 - |T|/(q(q-1))$ with the $|T|$ forbidden classes distinct. Its proof requires one genuinely new ingredient beyond Theorem~\ref{thm:singleform}: the M\"obius tails of the several forms must be controlled \emph{in mean} rather than pointwise, since pointwise they are divisor-type functions of size $N^{\varepsilon}$, which no logarithmic main term can absorb. The mean-value control is supplied by two elementary lemmas. For a shift $c \in \{1,2,4,8\}$, an integer $n > 8$, and a cutoff $1 \le y \le \infty$, write
\[
D_c(n) = \#\{d \ge 1 \text{ squarefree} : d^2 \mid n - c\},
\qquad
D_c^{>y}(n) = \#\{d > y \text{ squarefree} : d^2 \mid n - c\}.
\]

\begin{lemma}[Lattice sums]\label{lem:lattice}
For squarefree $L$, the number of quadruples $(d_1, d_2, d_3, d_4)$ of squarefree positive integers with $\operatorname{lcm}(d_1,\dots,d_4) = L$ is $15^{\omega(L)}$, and
\[
\sum_{L \ge 1} \frac{15^{\omega(L)}}{L^{3/2}} < \infty .
\]
Consequently $\sum_{L > y} 15^{\omega(L)} L^{-2} \ll y^{-1/2}$ and $\sum_{L \le Z} 15^{\omega(L)} \ll Z^{3/2}$.
\end{lemma}

\begin{proof}
A squarefree quadruple with lcm $L$ assigns to each prime $q \mid L$ the nonempty subset of indices $i$ with $q \mid d_i$: there are $2^4 - 1 = 15$ choices per prime. Since $15^{\omega}$ is multiplicative, $\sum_L 15^{\omega(L)} L^{-3/2} = \prod_q (1 + 15 q^{-3/2}) < \infty$. The tail bound follows from Rankin's trick, $\sum_{L>y} 15^{\omega}L^{-2} \le y^{-1/2} \sum_L 15^{\omega} L^{-3/2}$, and the count from $\sum_{L \le Z} 15^{\omega(L)} \le Z^{3/2} \sum_L 15^{\omega} L^{-3/2}$.
\end{proof}

\begin{lemma}[Fourth moments]\label{lem:moments}
Uniformly in the shift $c \in \{1,2,4,8\}$ and $2 \le y \le N$,
\[
\textup{(i)} \ \sum_{8 < n \le N} D_c(n)^4 \ll N,
\qquad
\textup{(ii)} \ \sum_{8 < n \le N} D_c^{>y}(n)^4 \ll \frac{N}{\sqrt{y}} + N^{3/4}.
\]
\end{lemma}

\begin{proof}
Expanding the fourth power and writing $L = \operatorname{lcm}(d_1, \dots, d_4)$,
\[
\sum_{8 < n \le N} D_c(n)^4 = \sum_{d_1, \dots, d_4} \#\{8 < n \le N : L^2 \mid n - c\}
\le \sum_{L} 15^{\omega(L)} \left(\frac{N}{L^2} + 1\right),
\]
where the sum may be restricted to $L^2 \le N$ (for $L^2 > N$ and $n > c$ the inner count is zero). By Lemma~\ref{lem:lattice} the $N/L^2$ part is $\ll N$ and the $+1$ part is $\sum_{L \le \sqrt{N}} 15^{\omega(L)} \ll N^{3/4}$, proving (i). For (ii) the quadruples satisfy $\min_i d_i > y$, hence $L > y$: the $N/L^2$ part is $\ll N \sum_{L > y} 15^{\omega(L)} L^{-2} \ll N/\sqrt{y}$, and the $+1$ part is again $\ll N^{3/4}$.
\end{proof}

\begin{theorem}[Squarefree equidistribution for several forms]\label{thm:sf}
For each subset $T$ of the four forms $\{x,\, 6x-1,\, 2x-1,\, 6x-7\}$, the set of primes $p \equiv 7 \pmod{36}$ for which \emph{all} forms in $T$ are squarefree has relative density
\[
\delta(T) = \prod_{q \ge 5} \left(1 - \frac{|T|}{q(q-1)}\right).
\]
\end{theorem}

\begin{proof}
Let $t = |T|$ and let $c_1 < \dots < c_t$ be the corresponding shifts (a subset of $\{1,2,4,8\}$; recall that $q^2 \mid f(x_p)$ if and only if $q^2 \mid p - c_f$ for $q \ge 5$, and no square of $2$ or $3$ divides any $p - c_i$, by Lemmas~\ref{lem:forms} and~\ref{lem:excluded}). Set $y = (\log N)^{16}$ and, for each $i$, split the squarefree indicator into a truncated M\"obius sum and a tail:
\[
\mathbf{1}[\,p - c_i \text{ squarefree}\,] = M_i(p) + R_i(p),
\qquad
M_i(p) = \sum_{\substack{d \le y \\ d^2 \mid p - c_i}} \mu(d),
\quad
R_i(p) = \sum_{\substack{d > y \\ d^2 \mid p - c_i}} \mu(d),
\]
with the pointwise bounds $|M_i| \le D_{c_i}$ and $|R_i| \le D_{c_i}^{>y}$. Summing the product over primes $p \le N$ of the class gives $S(N) = \mathcal{M} + \mathcal{E}$, where $\mathcal{M} = \sum_p \prod_i M_i(p)$ and $\mathcal{E}$ collects the $2^t - 1$ cross terms, each containing at least one tail factor.

\emph{Main term.} Expanding each $M_i$ produces a sum over tuples $(d_1, \dots, d_t)$ with all $d_i \le y$ of $\mu(d_1)\cdots\mu(d_t)$ times the number of primes $p \le N$ with $p \equiv 7 \pmod{36}$ and $p \equiv c_i \pmod{d_i^2}$ for all $i$. Tuples in which some $d_i$ shares a factor with $6$ contribute nothing (the congruence is incompatible with the class, as in Theorem~\ref{thm:singleform}). Tuples with $\gcd(d_i, d_j) > 1$ for some $i \ne j$ also contribute nothing: a prime $q \mid \gcd(d_i, d_j)$ would give $q^2 \mid p - c_i$ and $q^2 \mid p - c_j$, hence $q^2 \mid c_j - c_i$ with $0 < |c_j - c_i| \le 7 < 25 \le q^2$, which is impossible --- the elementary heart of Corollary~\ref{cor:separation}, and note that no special case is needed at $q = 7$, since the conditions sit on the shifted primes directly. For the surviving tuples --- pairwise coprime and coprime to $6$ --- the CRT merges the conditions into a single class modulo $m = 36\prod_i d_i^2$, admissible because each $c_i$ is a power of $2$ and each $d_i$ is odd. Since $m \le 36y^8$ is a fixed power of $\log N$, Siegel--Walfisz applies uniformly, and with at most $y^t$ tuples the accumulated error is $o(N/\log N)$. The completed coefficient sum factors as an Euler product --- at each prime $q \ge 5$, either $q$ divides no $d_i$ or exactly one --- giving
\[
\sum_{\substack{d_1, \dots, d_t \text{ pairwise coprime} \\ (d_i, 6) = 1}} \frac{\mu(d_1)\cdots\mu(d_t)}{\prod_i \varphi(d_i^2)}
= \prod_{q \ge 5}\left(1 - \frac{t}{q(q-1)}\right) = \delta(T),
\]
absolutely convergent, with truncation error $O(y^{-1/2})$. Hence $\mathcal{M} = \delta(T)\, \pi(N; 36, 7) + o(N/\log N)$.

\emph{Cross terms.} A cross term has the shape $\sum_p Z_1 \cdots Z_t$ with each $Z_i \in \{M_i, R_i\}$ and $Z_{i_0} = R_{i_0}$ for some $i_0$. Bounding $|M_i| \le D_{c_i}$, $|R_{i_0}| \le D_{c_{i_0}}^{>y}$, dropping the primality of $p$, padding the product to four factors by the constant function $1$, and applying H\"older with exponents $(4,4,4,4)$:
\begin{multline*}
\left| \sum_p Z_1 \cdots Z_t \right|
\le \left(\sum_{8 < n \le N} D_{c_{i_0}}^{>y}(n)^4\right)^{\!1/4}
\left(\prod_{i \ne i_0} \left(\sum_{8 < n \le N} D_{c_i}(n)^4\right)^{\!1/4}\right) N^{(4-t)/4}\\
\ll \left(\frac{N}{\sqrt y} + N^{3/4}\right)^{\!1/4} N^{3/4}
\end{multline*}
by Lemma~\ref{lem:moments}, which is $\ll N/y^{1/8} + N^{15/16} = N/(\log N)^2 + N^{15/16} = o(N/\log N)$. The theorem follows.
\end{proof}

\begin{remark}[Scope: larger windows]\label{rem:sfscope}
Nothing in the proof used $t \le 4$ beyond the padding in H\"older and the constant $15 = 2^4 - 1$: with $r$-th moments and $(2^r - 1)^{\omega}$, the identical argument covers any \emph{fixed} set of admissible exponents whose pairwise gaps are less than $20$, since the vanishing of non-coprime tuples then follows from Lemma~\ref{lem:ord} exactly as in Proposition~\ref{prop:extended}. In particular the five-form density $\rho_5$ used in Table~\ref{tab:exponents} is unconditional, as is $\rho_r$ for the first $r$ admissible exponents for every $r \le 16$ (the sixteenth admissible exponent is $a = 18$, so the window has width $18 < 20$). Beyond width $20$ the forbidden classes begin to coincide --- $2^{a+20} \equiv 2^a \pmod{25}$ --- and the local factors involve the number of distinct powers of $2$ modulo $q^2$, governed by $\operatorname{ord}_{q^2}(2)$; this is the Wieferich-type territory of \cite{gransound}, and we do not pursue it here.
\end{remark}

We write, for $0 \le j \le 5$,
\[
e_j := \prod_{q \ge 5} \left(1 - \frac{j}{q(q-1)}\right),
\]
so that $e_0 = 1$ and
\begin{align*}
e_1 &= 0.8974939526861\ldots, & e_2 &= 0.8026605189250\ldots, & e_3 &= 0.7150711323123\ldots,\\
e_4 &= 0.6343158637609\ldots, & e_5 &= 0.5600027699917\ldots
\end{align*}
Note $e_1$ is Artin's constant with the factors at $q = 2, 3$ removed. The digits are displayed to this length because the alternating sums below suffer severe cancellation; see the proof of Theorem~\ref{thm:conditional}.

\begin{theorem}[Main density theorem]\label{thm:conditional}
The set of primes $p \equiv 7 \pmod{36}$ that admit a representation $p = 2^a + m$ with $m$ squarefree and $a \in \{0, 1, 2, 3\}$ has relative density $1 - \rho_4$, where
\[
\rho_4 \;=\; \sum_{j=0}^{4} (-1)^j \binom{4}{j} e_j \;=\; 1.8637317202 \times 10^{-5} \;+\; O(10^{-14}).
\]
More generally, the set of primes for which the first $r$ admissible remainders all fail ($1 \le r \le 4$) has relative density $\rho_r = \sum_{j=0}^{r} (-1)^j \binom{r}{j} e_j$, with
\[
\rho_1 = 0.1025060473139\dots, \quad
\rho_2 = 0.0076726135528\dots, \quad
\rho_3 = 0.0004285664043\dots
\]
\end{theorem}

\begin{proof}
By Lemma~\ref{lem:forms}, $p$ admits a representation with $a \le 3$ if and only if at least one of the four forms is squarefree at $x_p$. By inclusion--exclusion over subsets $T$ of the forms,
\[
\text{density}(\text{at least one squarefree}) = \sum_{\emptyset \neq T} (-1)^{|T|-1} \delta(T)
= \sum_{j=1}^{4} (-1)^{j-1} \binom{4}{j} e_j ,
\]
using Theorem~\ref{thm:sf} and the fact that $\delta(T)$ depends only on $|T|$ (Lemma~\ref{lem:local}: the local factors are $1 - |T|/(q(q-1))$ for every $q \ge 5$ regardless of which forms lie in $T$, by distinctness of the forbidden classes). The complement has density $\rho_4 = \sum_{j=0}^4 (-1)^j \binom{4}{j} e_j$. The same computation with the first $r$ forms gives $\rho_r$.

A remark on the numerics is in order, since the alternating sum defining $\rho_4$ cancels terms of size $O(1)$ down to $2 \times 10^{-5}$ --- a loss of roughly five significant digits. The Euler products $e_j$ were therefore evaluated with $50$-digit working precision, as products over primes $q \le 10^6$ multiplied by tail corrections $\exp\bigl(\sum_{q > 10^6} \log(1 - j/(q(q-1)))\bigr)$ computed via prime zeta functions; the check $\prod_{q\ge5}(1 - q^{-2}) = 9/\pi^2$ reproduces $16$ digits. All digits displayed here and above are correct, and the stated error bound on $\rho_4$ is conservative.
\end{proof}

\begin{corollary}[Infinitude of complete failures]\label{cor:infinitude}
The number of complete failures $p \le N$ is asymptotic to $\rho_4 \cdot \pi(N; 36, 7)$ with $\rho_4 = 1.8637317202\times 10^{-5} > 0$. In particular there are infinitely many primes $p \equiv 7 \pmod{36}$ admitting no representation $p = 2^a + m$ with $m$ squarefree and $a \le 3$; the first eleven are listed in Table~\ref{tab:failures}.
\end{corollary}

\begin{proof}
Immediate from Theorem~\ref{thm:conditional}, since a complete failure is precisely a prime of the class admitting no such representation with $a \le 3$ (Definition~\ref{def:completefailure}), and $\pi(N; 36, 7) \to \infty$.
\end{proof}

\begin{remark}[Density, not universality]\label{rem:density}
Theorem~\ref{thm:conditional} asserts that the exceptional set is \emph{sparse}, not that it is empty, and the distinction is genuine: complete failures exist, the smallest being $p = 4{,}851{,}151$ (with $x = 808{,}525$ and obstructing squares $5^2 \mid x$, $409^2 \mid 6x-1$, $7^2 \mid 2x-1$, $53^2 \mid 6x-7$), and by Corollary~\ref{cor:infinitude} there are infinitely many. Section~\ref{sec:computation} exhibits all eleven below $10^8$.
\end{remark}

Over integer arguments, where the local densities are $1/q^2$, no hypothesis is needed at all: strong separation collapses the problem into linear pieces, and the whole analysis is unconditional.

\begin{theorem}[Integer arguments, unconditional]\label{thm:integer}
For each subset $T$ of the four forms, the set of integers $x \equiv 1 \pmod 6$, $x \le X$, with all forms in $T$ squarefree has density (relative to such $x$) equal to $\prod_{q \ge 5}(1 - |T|/q^2)$. Consequently, the set of $x \equiv 1 \pmod 6$ at which all four forms are non-squarefree has density
\[
\rho_4^{\mathbb{Z}} = \sum_{j=0}^{4} (-1)^j \binom{4}{j}\, d_j = 1.2018874348 \times 10^{-5}, \qquad d_j := \prod_{q \ge 5}\left(1 - \frac{j}{q^2}\right),
\]
where
\[
d_1 = \tfrac{9}{\pi^2} = 0.9118906527810\ldots, \quad d_2 = 0.8296305401295\ldots,
\]
\[
d_3 = 0.7529218854349\ldots, \quad d_4 = 0.6814789309610\ldots
\]
\end{theorem}

\begin{proof}
Fix $T$ and a parameter $z \ge 7$, and let $A_z(X)$ count $x \le X$, $x \equiv 1 \pmod 6$, such that no prime $q \in [5, z]$ has $q^2 \mid f(x)$ for some $f \in T$. For a single prime $q \ge 5$, the condition $q^2 \mid f(x)$ is a single residue class for $x$ modulo $q^2$ (the leading coefficient of each form is invertible modulo $q$), and by Corollary~\ref{cor:separation} the classes attached to distinct forms in $T$ are \emph{distinct} --- the only delicacy is $q = 7$, where $7$ may divide both $x$ and $6x - 7$, but squarefreeness concerns $49$, and the two forbidden classes $x \equiv 0$ and $x \equiv 42 \pmod{49}$ are distinct by the final clause of Theorem~\ref{thm:coprime}. Hence exactly $|T|$ classes modulo $q^2$ are forbidden at each $q \in [5, z]$, and by the Chinese Remainder Theorem,
\[
A_z(X) = \frac{X}{6} \prod_{5 \le q \le z} \left(1 - \frac{|T|}{q^2}\right) + O\!\left(\prod_{5 \le q \le z} q^2\right).
\]
The count $N_T(X)$ of $x$ with all forms in $T$ squarefree differs from $A_z(X)$ only by $x$ suffering an obstruction $q^2 \mid f(x)$ with $q > z$; since $f(x) \le 6X + O(1)$, there are at most
$\sum_{f \in T} \sum_{q > z} \left( X/q^2 + 1 \right) \ll X/z + \sqrt{X}$
such $x$. Choose $z = \tfrac14 \log X$: then the accumulated error above is $\prod_{5 \le q \le z} q^2 \le e^{2\vartheta(z)} \le X^{3/5}$ for large $X$ by Chebyshev's bound $\vartheta(z) \ll z$, while the tail contributes $O(X/\log X)$; hence $N_T(X) = \tfrac{X}{6}\prod_{q \ge 5}(1 - |T|/q^2) + o(X)$, which is the first claim. (The cutoff must not be taken larger: for $z$ of order $\log X$ with a large constant, the product $e^{2\vartheta(z)}$ would overwhelm the main term.) Inclusion--exclusion over $T$ as in Theorem~\ref{thm:conditional} gives the second. The constants $d_j$ and $\rho_4^{\mathbb{Z}}$ were computed exactly as in Theorem~\ref{thm:conditional}, with $d_1 = 9/\pi^2$ recovered to $16$ digits as a check.
\end{proof}

\begin{remark}[Where $abc$ would enter]\label{rem:abcrole}
Theorem~\ref{thm:integer} uses nothing beyond Corollary~\ref{cor:separation} and elementary sieving, even though $F_T = \prod_{f \in T} f$ has degree up to $4$: separation collapses the square-divisor problem for the quartic into independent \emph{linear} problems, where the tail $\sum_{q > z}(X/q^2 + 1)$ is trivial to control. There is no genuinely quartic behaviour, hence no range of large $q$ in which the elementary method fails. We note that separation is a convenience rather than a necessity here: for any product of distinct linear forms, all but finitely many $q$ have $q^2$ dividing at most one factor, so the same elementary argument applies after excluding the divisors of the pairwise resultants.
\end{remark}

\begin{remark}
The two constants differ by more than 50\%: $\rho_4^{\mathbb{Z}} \approx 1.20 \times 10^{-5}$ versus $\rho_4 \approx 1.86 \times 10^{-5}$. Restricting to prime arguments makes complete failures \emph{more} likely, because the conditions $q^2 \mid p - 2^a$ concentrate on the reduced residue classes. The data of Section~\ref{sec:computation} decisively favor the prime-corrected constant. Both are far smaller than the naive independence estimate $(1 - e_1)^4 \approx 1.10 \times 10^{-4}$ (or $(1 - d_1)^4 \approx 6.03 \times 10^{-5}$ with $d_1 = \prod_{q \ge 5}(1 - q^{-2}) = 9/\pi^2$), which ignores the mutual exclusivity of Corollary~\ref{cor:separation}: a single small prime such as $q = 5$ can account for only one of the four failures, and this negative correlation suppresses the joint event by a factor of about six.
\end{remark}

Finally, we record what can and cannot be said about the full conjecture, with all exponents available. The two statements at issue, Conjecture~\ref{conj:class} and its cofiniteness form Conjecture~\ref{conj:cofinite}, are stated together in Section~\ref{sec:heuristic}; neither is within reach of our methods, and the obstruction is the same in both cases.

\begin{remark}[How far the density method reaches]\label{rem:cofinite}
Theorem~\ref{thm:sf} in the extended form of Remark~\ref{rem:sfscope} gives, for each \emph{fixed} $r \le 16$, the exact relative density $\rho_r$ of primes failing all of the first $r$ admissible remainders, unconditionally. Taking $r = 16$ (admissible exponents $a \in \{0,\dots,18\} \setminus \{4, 10, 16\}$, a window of width $18 < 20$) yields an unconditional bound on the exceptional set of Conjecture~\ref{conj:cofinite}: it has relative density at most
\[
\rho_{16} = \sum_{j=0}^{16} (-1)^j \binom{16}{j} e_j = 4.2809 \times 10^{-27}.
\]
The cancellation here is far more severe than in Theorem~\ref{thm:conditional}: the largest term is $\binom{16}{8}e_8 \approx 4.8 \times 10^{3}$ against an answer of order $10^{-27}$, so roughly thirty significant digits are lost. This value was computed with $60$-digit working precision, and it is stable both under an increase of the precision to $100$ digits and under an increase of the Euler-product truncation from $10^6$ to $2.5 \times 10^6$.

One would like to let $r$ grow and conclude that the exceptional set has relative density zero. We have not established this, but we do not think the obstruction is where Remark~\ref{rem:sfscope} might suggest. The restriction to windows of width less than $20$ there is an artefact of how the local factors were stated, not of the method: for an arbitrary finite set $A$ of admissible exponents, the conditions $q^2 \mid p - 2^a$ for $a \in A$ single out the residues $2^a \bmod q^2$, and two exponents give the same condition precisely when they agree modulo $\operatorname{ord}_{q^2}(2)$. Coincident conditions merge consistently under the Chinese Remainder Theorem rather than conflicting, so the argument of Theorem~\ref{thm:sf} goes through with local factor
\[
1 - \frac{\#\{\,2^a \bmod q^2 : a \in A\,\}}{\varphi(q^2)},
\]
which agrees with $1 - |A|/\varphi(q^2)$ for all but the finitely many $q$ with $\operatorname{ord}_{q^2}(2) \le \max A$. Writing $\rho_A$ for the resulting density of primes failing every $a \in A$, the exceptional set of Conjecture~\ref{conj:cofinite} has relative density at most $\rho_A$ for every such $A$, and density zero would follow from $\rho_A \to 0$ as $A$ exhausts the admissible exponents. We have not proved that limit. It is not merely a matter of bookkeeping: for a given $p$, the exponents $a$ that are obstructed by a fixed $q$ form a congruence class modulo $\operatorname{ord}_{q^2}(2)$, so $\rho_A \to 0$ amounts to the assertion that the admissible exponents admit no covering by such classes, a question of covering-system type. Even density zero would then be a very long way from finiteness --- the squares have density zero --- and finiteness would require counting bounds uniform in $|A|$, which we do not have. We therefore state cofiniteness as a conjecture.
\end{remark}

\section{Computational Evidence}\label{sec:computation}

\subsection{Method}

We verified Conjecture~\ref{conj:onepower} for all primes $p \equiv 7 \pmod{36}$ up to $N = 10^8$ by direct sieving, in three stages: (i) a squarefree sieve on $[1, 10^8]$, marking multiples of $q^2$ for all primes $q \le 10^4$; (ii) a prime sieve extracting the $\pi(10^8; 36, 7) = 480{,}164$ primes in the class; (iii) for each such $p$, table lookups deciding squarefreeness of $p - 2^a$ for increasing admissible $a$. All arithmetic is exact 64-bit integer arithmetic; the obstructing prime squares reported below were recomputed independently by trial division as a consistency check, and the previously known failure $p = 4{,}851{,}151$ (with its obstruction quadruple $(5, 409, 7, 53)$) was reproduced. The full computation takes under one minute in vectorized Python (NumPy) on commodity hardware.

\subsection{Results}

\begin{theorem}[Computational verification]\label{thm:computational}
Every prime $p \equiv 7 \pmod{36}$ with $p < 10^8$ can be written as $p = 2^a + m$ with $m$ squarefree and $a \in \{0, 1, 2, 3, 5\}$. Exactly eleven such primes are complete failures (no admissible $a \le 3$ works); each of the eleven is resolved at $a = 5$, i.e., $p - 32 = 6x - 31$ is squarefree. The exponent $a = 4$ is never usable (Lemma~\ref{lem:excluded}).
\end{theorem}

Table~\ref{tab:decades} shows the growth of the class and of its failures by decade, together with the expected counts $\rho_1 \cdot \pi(N; 36, 7)$ and $\rho_4 \cdot \pi(N; 36, 7)$ from Theorem~\ref{thm:conditional}.

\begin{table}[ht]
\centering
\begin{tabular}{rrrrrr}
\toprule
$N$ & $\pi(N;36,7)$ & primary & $\rho_1 \pi(N;36,7)$ & complete & $\rho_4 \pi(N;36,7)$ \\
\midrule
$10^4$ & $106$    & $11$     & $10.9$     & $0$  & $0.002$ \\
$10^5$ & $804$    & $87$     & $82.4$     & $0$  & $0.015$ \\
$10^6$ & $6{,}549$   & $674$    & $671.3$    & $0$  & $0.122$ \\
$10^7$ & $55{,}397$  & $5{,}633$   & $5{,}678.5$   & $1$  & $1.03$ \\
$10^8$ & $480{,}164$ & $49{,}265$  & $49{,}219.7$  & $11$ & $8.95$ \\
\bottomrule
\end{tabular}

\caption{Primary failures ($x$ non-squarefree) and complete failures (all four forms non-squarefree) among primes $p \equiv 7 \pmod{36}$, $p \le N$, with model predictions.}
\label{tab:decades}
\end{table}

The individual forms fail at nearly identical rates, as the model predicts: among the $480{,}164$ primes of the class below $10^8$, the counts of non-squarefree values are $49{,}265$ for $x$, $49{,}101$ for $6x-1$, $49{,}131$ for $2x-1$, and $49{,}171$ for $6x-7$, each within $0.3\%$ of the predicted $\rho_1 \cdot 480{,}164 = 49{,}219.7$.

Table~\ref{tab:failures} lists the eleven complete failures with their obstruction quadruples; in every row the four obstructing primes are distinct and at least $5$, as Proposition~\ref{prop:poly} requires.

\begin{table}[ht]
\centering
\begin{tabular}{rrrrrrr}
\toprule
$p$ & $x = (p-1)/6$ & $q_1$ & $q_2$ & $q_3$ & $q_4$ & $\min_i q_i$ \\
\midrule
$4{,}851{,}151$  & $808{,}525$    & $5$   & $409$ & $7$   & $53$ & $5$ \\
$14{,}627{,}383$ & $2{,}437{,}897$  & $7$   & $31$  & $23$  & $5$  & $5$ \\
$17{,}026{,}927$ & $2{,}837{,}821$  & $19$  & $5$   & $23$  & $13$ & $5$ \\
$19{,}883{,}527$ & $3{,}313{,}921$  & $13$  & $5$   & $23$  & $19$ & $5$ \\
$50{,}671{,}051$ & $8{,}445{,}175$  & $5$   & $11$  & $7$   & $19$ & $5$ \\
$63{,}716{,}551$ & $10{,}619{,}425$ & $5$   & $13$  & $53$  & $11$ & $5$ \\
$64{,}734{,}883$ & $10{,}789{,}147$ & $31$  & $19$  & $11$  & $5$  & $5$ \\
$72{,}319{,}651$ & $12{,}053{,}275$ & $5$   & $17$  & $109$ & $11$ & $5$ \\
$95{,}075{,}683$ & $15{,}845{,}947$ & $13$  & $37$  & $41$  & $5$  & $5$ \\
$96{,}041{,}527$ & $16{,}006{,}921$ & $149$ & $5$   & $19$  & $7$  & $5$ \\
$98{,}555{,}227$ & $16{,}425{,}871$ & $11$  & $5$   & $13$  & $7$  & $5$ \\
\bottomrule
\end{tabular}

\caption{All complete failures $p < 10^8$: the four candidate remainders $p-1, p-2, p-4, p-8$ are non-squarefree, obstructed by $q_1^2 \mid x$, $q_2^2 \mid 6x{-}1$, $q_3^2 \mid 2x{-}1$, $q_4^2 \mid 6x{-}7$. In each case $p - 32$ is squarefree.}
\label{tab:failures}
\end{table}

The number of complete failures is itself a check on Theorem~\ref{thm:conditional}: with $\lambda = \rho_4 \cdot 480{,}164 = 8.95$, observing eleven is unremarkable, since $\mathbb{P}(X \ge 11) \approx 0.29$ for $X \sim \mathrm{Poisson}(\lambda)$.

Notably, $\min_i q_i = 5$ in every row: each of the eleven complete failures employs the most probable obstructing prime somewhere among its four slots. This is consistent with Table~\ref{tab:obstructors} below, where $q = 5$ accounts for $49\%$ of all primary failures, and it makes the connection between the two tables visible: complete failures are assembled from the same obstruction statistics that govern single forms.

Table~\ref{tab:exponents} gives the distribution of the smallest usable exponent, compared with the model: the probability that the smallest admissible exponent is the $(r+1)$-st admissible value is $\rho_r - \rho_{r+1}$ (with $\rho_0 = 1$ and $\rho_5$ given by the five-form case of Theorem~\ref{thm:sf}, via Proposition~\ref{prop:extended} and Remark~\ref{rem:sfscope}). Every entry in the ``expected'' column is therefore a theorem, not a heuristic.

\begin{table}[ht]
\centering
\begin{tabular}{crr}
\toprule
smallest $a$ & observed & expected \\
\midrule
$0$ & $430{,}899$ & $430{,}944.3$ \\
$1$ & $45{,}545$  & $45{,}535.6$  \\
$2$ & $3{,}527$   & $3{,}478.3$   \\
$3$ & $182$    & $196.8$    \\
$4$ & $0$      & $0$        \\
$5$ & $11$     & $8.6$      \\
$>5$ & $0$     & $0.3$      \\
\bottomrule
\end{tabular}

\caption{Distribution of the smallest exponent $a$ with $p - 2^a$ squarefree, over the $480{,}164$ primes $p \equiv 7 \pmod{36}$, $p < 10^8$. The entry $a = 4$ vanishes identically by Lemma~\ref{lem:excluded}.}
\label{tab:exponents}
\end{table}

Finally, Table~\ref{tab:obstructors} examines \emph{which} prime obstructs the primary form: for each of the $49{,}265$ primes with $x$ non-squarefree, we record the smallest prime $q$ with $q^2 \mid x$. The prime-corrected model predicts that $q$ is the smallest obstructor with probability
\[
\frac{1}{q(q-1)} \prod_{5 \le r < q} \left(1 - \frac{1}{r(r-1)}\right);
\]
the naive integer model replaces $q(q-1)$ by $q^2$. The data discriminate sharply between the two.

\begin{table}[ht]
\centering
\begin{tabular}{crrr}
\toprule
$q$ & observed & prime-corrected & naive \\
\midrule
$5$  & $24{,}039$ & $24{,}008.2$ & $19{,}206.6$ \\
$7$  & $10{,}835$ & $10{,}860.9$ & $9{,}407.3$  \\
$11$ & $4{,}025$  & $4{,}048.1$  & $3{,}731.8$  \\
$13$ & $2{,}852$  & $2{,}828.5$  & $2{,}649.8$  \\
$17$ & $1{,}629$  & $1{,}611.8$  & $1{,}540.4$  \\
$19$ & $1{,}266$  & $1{,}277.2$  & $1{,}228.9$  \\
$23$ & $874$   & $860.7$   & $836.3$   \\
$29$ & $543$   & $535.3$   & $525.0$   \\
$31$ & $484$   & $466.8$   & $458.9$   \\
$37$ & $332$   & $325.6$   & $321.8$   \\
\bottomrule
\end{tabular}

\caption{Smallest obstructing prime $q$ (i.e., smallest $q$ with $q^2 \mid x$) among the $49{,}265$ primary failures below $10^8$: observed counts against the prime-corrected model ($1/q(q-1)$ local densities) and the naive model ($1/q^2$).}
\label{tab:obstructors}
\end{table}

The overall primary failure rate makes the comparison starkest, and it can be made precise. The observed rate is $49{,}265/480{,}164 = 0.102600$; with $N = 480{,}164$ trials, the binomial standard deviation of the rate is $\sqrt{\rho_1(1-\rho_1)/N} \approx 4.38 \times 10^{-4}$. The prime-corrected prediction $\rho_1 = 0.102506$ lies $0.22\sigma$ from the observation; the naive prediction $1 - 9/\pi^2 = 0.088109$ lies $33.1\sigma$ away. For Table~\ref{tab:obstructors} as a whole, the $\chi^2$ statistic over the ten displayed categories is $1.79$ for the prime-corrected model against $1481$ for the naive model. The expected counts here are not normalized to the observed number of primary failures; they are the model applied to all $480{,}164$ primes of the class, and they sum to $46{,}823$ against the $46{,}879$ observed in those categories. Extending the ten categories to the full partition of the class, by adjoining a category for obstruction by some $q \ge 41$ and a category for $x$ squarefree, gives $\chi^2 = 1.84$ on $11$ degrees of freedom. A value that far below its expectation corresponds to $p \approx 0.999$, so the agreement is if anything slightly better than chance alone would produce. The naive model is excluded; the prime-corrected model fits at every observable depth.

\section{The Heuristic Model and Conjectures}\label{sec:heuristic}

The model underlying Theorem~\ref{thm:conditional} and the tables above is the following. For a random prime $p \equiv 7 \pmod{36}$ and each prime $q \ge 5$, the events $q^2 \mid p - 2^a$ over admissible exponents $a$ in a window of width less than $20$ are mutually exclusive (Proposition~\ref{prop:extended} and Remark~\ref{rem:wieferich}), each of relative density $1/\varphi(q^2) = 1/(q(q-1))$ (Lemma~\ref{lem:local}); events at distinct primes $q$ are asymptotically independent, which for fixed windows is the content of Theorem~\ref{thm:sf}. This parallels the heuristic that Hercher \cite{hercher} develops for the odd-integer conjecture, with two differences: our local densities carry the prime-argument correction $q^2 \to q(q-1)$, and in our residue class the modulus-$9$ obstruction is absorbed structurally by Lemma~\ref{lem:excluded} (exponents $a \equiv 4 \bmod 6$ are excluded outright) rather than entering the local factors, so all remaining local computations run over $q \ge 5$ only.

For every \emph{fixed} exponent window, this model is now a theorem: Theorem~\ref{thm:sf} and Remark~\ref{rem:sfscope} justify the local factors, and Theorem~\ref{thm:conditional} and Corollary~\ref{cor:infinitude} convert them into densities --- the asymptotic count $\rho_4 \cdot \pi(N; 36,7)$ of complete failures is proved, not conjectured, and the agreement in Tables~\ref{tab:decades}--\ref{tab:obstructors} is a check on the theorems rather than evidence for a guess. What remains genuinely conjectural is the behavior when the window is allowed to grow with $p$, where every uniformity we would need fails (Remark~\ref{rem:cofinite}). We therefore state the two conjectures that remain, in decreasing order of strength.

\begin{conjecture}[One-power conjecture for the class]\label{conj:class}
Every prime $p \equiv 7 \pmod{36}$ can be written as $p = 2^a + m$ with $a \ge 0$, $a \not\equiv 4 \pmod 6$, and $m$ squarefree.
\end{conjecture}

We state also the weaker form, which our methods likewise do not reach.

\begin{conjecture}[Cofiniteness]\label{conj:cofinite}
All but finitely many primes $p \equiv 7 \pmod{36}$ can be written as $p = 2^a + m$ with $a \ge 0$ and $m$ squarefree.
\end{conjecture}

Conjecture~\ref{conj:class} is Conjecture~\ref{conj:onepower} for the class, sharpened by Lemma~\ref{lem:excluded}; Theorem~\ref{thm:computational} verifies it to $10^8$ with $a \le 5$, and Conjecture~\ref{conj:cofinite} is its at-most-finitely-many-exceptions form (see Remark~\ref{rem:cofinite} for why even $abc$ does not appear to deliver this). In the other direction, no bounded set of exponents can suffice permanently:

\begin{proposition}[The required exponent is unbounded]\label{prop:unbounded}
For every $L$, there exist infinitely many primes $p \equiv 7 \pmod{36}$ such that $p - 2^a$ is non-squarefree for every $a \le L$. Consequently $\limsup_p a_{\min}(p) = \infty$, where $a_{\min}(p)$ is the smallest usable exponent.
\end{proposition}

\begin{proof}
Let $A = \{a \le L : a \not\equiv 4 \pmod 6\}$ be the admissible exponents up to $L$, and choose distinct primes $q_a \ge 5$, one for each $a \in A$. The congruences
\[
p \equiv 7 \pmod{36}, \qquad p \equiv 2^a \pmod{q_a^2} \quad (a \in A)
\]
have pairwise coprime moduli, and each is a class coprime to its modulus ($2^a$ is coprime to $q_a$). By CRT they determine a single admissible class modulo $36\prod_a q_a^2$, which by Dirichlet's theorem contains infinitely many primes. Discard the finitely many such primes with $p \le 2^L$, so that $p - 2^a$ is a positive integer for every $a \le L$; infinitely many remain. For each of these, $q_a^2 \mid p - 2^a$ for every admissible $a \le L$, while $9 \mid p - 2^a$ for the inadmissible $a$; hence no $a \le L$ is usable.
\end{proof}

This is the class-restricted analogue of the observation of Hercher \cite[Remark 4]{hercher} and, like his, shows that computational verification can never terminate at a bounded exponent; compare also the Wieferich-prime connection of Granville and Soundararajan \cite{gransound}, which enters exactly when exponent windows of width $\ge 20$ are needed (Remark~\ref{rem:wieferich}).

\section{Conclusion}\label{sec:conclusion}

For primes $p \equiv 7 \pmod{36}$, the one-power Erd\H{o}s conjecture reduces to the squarefreeness of the values of a handful of linear forms in the single variable $x = (p-1)/6$, and the pairwise-coprimality theorem turns the failure analysis into a clean question about square divisors of the squarefree quartic $P(x) = x(6x-1)(2x-1)(6x-7)$. We proved the strong separation theorem, extended it to every window of admissible exponents of width less than $20$, and identified the Wieferich-type boundary at exponent gaps of exactly $20$. Separation does real work: it makes the complete-failure density over integer arguments an unconditional theorem (Theorem~\ref{thm:integer}), and it is what makes the prime-argument analysis go through --- combined with Mirsky's method for single forms (Theorem~\ref{thm:singleform}) and, for several forms at once, with a truncated M\"obius expansion whose cross terms are controlled in mean by fourth-moment estimates (Theorem~\ref{thm:sf}). Consequently the density $1 - \rho_4$ of representable primes (Theorem~\ref{thm:conditional}) and the infinitude of complete failures (Corollary~\ref{cor:infinitude}) are unconditional, and every ``expected'' entry in our tables is a proved asymptotic. Computationally, we verified the conjecture for the class to $10^8$ with exponents in $\{0,1,2,3,5\}$, determined all eleven complete failures in that range, and matched the data against a prime-corrected heuristic whose predictions are accurate at every observable depth --- the correction $1/q^2 \to 1/\varphi(q^2)$, forced by Dirichlet's theorem, is visible in the data at $33$ standard deviations. Appendix~\ref{sec:hercher} corrects the error bound in \cite[Lemma 2]{hercher}, now with a complete proof, verifying that its conclusions stand.

Several directions remain open. First, and most substantially, the gap between our density theorems and Conjecture~\ref{conj:class} is entirely a question of \emph{uniformity in the window size}: we can bound the exceptional set by $\rho_{16} < 10^{-26}$ but not by zero, because every error term in Section~\ref{sec:conditional} degrades as the number of forms grows, and past width $20$ even the local factors change shape (Remark~\ref{rem:sfscope}). A version of Theorem~\ref{thm:sf} with error terms uniform in $t$, valid for $t$ growing with $N$, would be the decisive advance; the Wieferich-type obstruction of \cite{gransound} suggests this is genuinely hard. Second, a quantitative \emph{upper} bound for complete failures via a direct sieve treatment of quadruples $(q_1q_2q_3q_4)^2 \mid P(x)$ --- rather than through inclusion--exclusion --- would likely be more robust to growing windows, and is worth attempting. Third, the feature of the class $p \equiv 7 \pmod{36}$ that the argument really uses is the behaviour at $2$ and $3$: the admissible exponents are exactly those with $a \not\equiv 4 \pmod 6$, and for these the content of $p - 2^a$ is squarefree, so that obstructions come only from $q \ge 5$. Classifying the pairs $(r,m)$ for which the analogous statement holds, and for which the admissible exponents form as convenient a set, would delimit where the method applies; the linearization itself is not the operative feature, since $p - 2^a$ is linear in $p$ in every class. Finally, extending the computation to $10^{10}$ or beyond --- where Corollary~\ref{cor:infinitude} predicts roughly $700$ complete failures below $10^{10}$ --- would test the asymptotic in a regime where the error term, not the main term, dominates our uncertainty, and searching for a prime requiring exponent $a \ge 6$ (none exists below $10^8$; the model expects the first around a few times $10^{8}$) would probe the tail of Table~\ref{tab:exponents}.

\appendix
\section{A Correction to a Lemma of Hercher}\label{sec:hercher}

While reconstructing the proofs in \cite{hercher}, we found that \cite[Lemma 2]{hercher} is incorrect as stated. We record the issue and its (harmless) consequences here, both for completeness and because Section~\ref{sec:heuristic} relies on Euler products of the same shape $\prod_{q \ge 5}(1 - j/q(q-1))$, computed by the same technique.

The lemma asserts: for real $0 < x \le \tfrac19$ and integer $2 \le k \le 20$ with $kx \le \tfrac45$,
\[
\log(1 - kx) = k \log(1-x) + \binom{k}{2} \log(1 - x^2) + e(x), \qquad |e(x)| < 18 x^3.
\]
The bound $|e(x)| < 18x^3$ cannot hold. Expanding each logarithm in Taylor series, the coefficient of $x^3$ in $e(x)$ is
\[
-\frac{k^3}{3} + \frac{k}{3} = -\frac{k^3 - k}{3},
\]
so that $e(x) = -\tfrac{k^3-k}{3}x^3 + O_k(x^4)$ as $x \to 0^+$. For $k = 20$ the cubic coefficient is $-\tfrac{8000 - 20}{3} = -2660$, so $|e(x)| \sim 2660\,x^3$ already for small $x$, exceeding $18x^3$ for every sufficiently small $x > 0$ (indeed, the claimed bound fails for all $k \ge 4$). The slip in the printed proof is that the Lagrange remainder of $\log(1 - kx)$ is bounded with an absolute constant, dropping the factor $k^3$.

\begin{lemma}[Corrected form of {\cite[Lemma 2]{hercher}}]\label{lem:herchercorrected}
For real $0 < x \le \tfrac{1}{25}$ and integer $2 \le k \le 20$,
\[
\left| \log(1 - kx) - k \log(1-x) - \binom{k}{2} \log(1 - x^2) \right| \;<\; 7637\, x^3,
\]
and the constant is sharp up to rounding: $\sup_{0 < x \le 1/25} |e(x)|/x^3 = 7636.79\ldots$, attained as $x \to \tfrac1{25}$ at $k = 20$.
\end{lemma}

\begin{proof}
Combining the three logarithm series (all convergent for $0 < x < 1/k$, hence on our domain, where also $kx \le \tfrac45$),
\[
-e(x) \;=\; \sum_{n \ge 1} \frac{k^n - k}{n}\, x^n \;-\; \binom{k}{2} \sum_{m \ge 1} \frac{x^{2m}}{m}.
\]
The $n = 1$ term vanishes and the $n = 2$ term cancels against the $m = 1$ term, leaving
\[
-e(x) \;=\; \sum_{n \ge 3} c_n(k)\, x^n,
\qquad
c_n(k) = \begin{cases} \dfrac{k^n - k}{n}, & n \text{ odd},\\[2mm] \dfrac{k^n - k^2}{n}, & n \text{ even},\end{cases}
\]
since for even $n = 2m$ the two contributions merge into $\bigl(k^n - k - 2\binom{k}{2}\bigr)/n = (k^n - k^2)/n$. Every coefficient $c_n(k)$ is nonnegative for $k \ge 2$ and $n \ge 3$, and each is nondecreasing in $k$. Therefore $e(x) < 0$, and
\[
\frac{|e(x)|}{x^3} \;=\; \sum_{n \ge 3} c_n(k)\, x^{n-3}
\]
is a sum of nonnegative nondecreasing functions of $x$ on $(0, \tfrac1{25}]$, hence nondecreasing in $x$, and nondecreasing in $k$ as well. The supremum over $2 \le k \le 20$ and $0 < x \le \tfrac1{25}$ is thus the single value at $(k, x) = (20, \tfrac1{25})$:
\[
25^3 \left| \log\tfrac15 - 20 \log\tfrac{24}{25} - 190 \log\bigl(1 - \tfrac{1}{625}\bigr) \right| \;=\; 7636.7967\ldots \;<\; 7637. \qedhere
\]
\end{proof}

\medskip
\noindent
\textit{Impact on \cite{hercher}.} The lemma is applied in \cite[Lemma 3]{hercher} with $x = p^{-2}$ over primes $p > m > 5$; thus $x \le 49^{-1} < 25^{-1}$ in general, and Lemma~\ref{lem:herchercorrected} applies. In the computation reported in \cite{hercher} the truncation is $m = 10^7$, so in fact $x < 10^{-14}$ throughout. Propagating the corrected constant, the error factor $f$ there satisfies
\[
|\log f| \;<\; 7637 \sum_{p > m} p^{-6} \;<\; \frac{7637}{5}\, m^{-5} \;<\; 1530\, m^{-5}
\]
in place of the printed $4m^{-5}$. At $m = 10^7$ this is below $1.6 \times 10^{-32}$, utterly negligible at the 40-digit working precision of his computation. \emph{All numerical conclusions of \cite{hercher} --- in particular his Theorem~5 and Tables~1, 2, and 4 --- are unaffected.} We have communicated the correction to the author.

\end{document}